\documentclass{amsart}
\usepackage{amsmath,amsthm,amsfonts,amssymb,latexsym,mathrsfs,graphicx}
\usepackage{pgf,tikz}
\usepackage{mathrsfs}
\usetikzlibrary{arrows}
\usepackage{tabularx}
\usepackage{hyperref}

\usepackage{enumerate}
\usepackage[shortlabels]{enumitem}
\usepackage{color}

\usepackage{pstricks-add}
\usepackage{graphicx}
\usepackage{pgf,tikz,pgfplots}
\usepackage{mathrsfs}
\usetikzlibrary{arrows}
\definecolor{ududff}{rgb}{0.30196078431372547,0.30196078431372547,1}

\newtheorem{theorem}{Theorem}[section]

\newtheorem{lemma}[theorem]{Lemma}

\theoremstyle{definition}

\newtheorem{rem}[theorem]{Remark}

\newtheorem*{ThmA}{Theorem A}

\newenvironment{enumeratei}{\begin{enumerate}[\upshape (a)]}
    {\end{enumerate}}

\def\irr#1{{\rm Irr}(#1)}
\def\cent#1#2{{\bf C}_{#1}(#2)}

\def\aut#1{{\rm Aut}(#1)}

\def\irr#1{{\rm Irr}(#1)}

\def\cent#1#2{{\bf C}_{#1}(#2)}

\def\norm#1#2{{\bf N}_{#1}(#2)}

\def\aut#1{{\rm Aut}(#1)}

\def\sym#1{{\rm Sym}(#1)}

\def\SL#1{{\rm SL}_{2}(#1)}
\def\PSL#1{{\rm PSL}_{2}(#1)}

\def\E#1{{\bf E}(#1)}

\def\irr#1{{\rm Irr}(#1)}

\def\cent#1#2{{\bf C}_{#1}(#2)}

\def\ker#1{{\rm ker}(#1)}
\def\norm#1#2{{\bf N}_{#1}(#2)}

\mathchardef\coso="2023

\begin{document}

\title[Element orders and codegrees of characters in non-solvable groups]{Element orders and codegrees of characters\\ in non-solvable groups}

\author[Z. Akhlaghi et al.]{Zeinab Akhlaghi}
\address{Zeinab Akhlaghi, Faculty of Math. and Computer Sci., \newline Amirkabir University of Technology (Tehran Polytechnic), 15914 Tehran, Iran.\newline
School of Mathematics,
Institute for Research in Fundamental Science(IPM)
P.O. Box:19395-5746, Tehran, Iran.}
\email{z\_akhlaghi@aut.ac.ir}


\author[]{Emanuele Pacifici}
\address{Emanuele Pacifici, Dipartimento di Matematica e Informatica U. Dini,\newline
Universit\`a degli Studi di Firenze, viale Morgagni 67/a,
50134 Firenze, Italy.}
\email{emanuele.pacifici@unifi.it}

\author[]{Lucia Sanus}
\address{Lucia Sanus, Departament de Matem\`atiques, Facultat de
Matem\`atiques, \newline
Universitat de Val\`encia,
46100 Burjassot, Val\`encia, Spain.}
\email{lucia.sanus@uv.es}

\thanks{The first author is supported by a grant from IPM (No. 1402200112), the second author is partially supported by INdAM-GNSAGA, and the third author is partially supported by the
Spanish Ministerio de Ciencia e Innovaci\'on
(Grant PID2019-103854GB-I00 funded by MCIN/AEI/10.13039/501100011033) and by Generalitat Valenciana CIAICO/2021/163}


\begin{abstract} 
 Given a finite group $G$ and an irreducible complex character $\chi$ of $G$, the \emph{codegree} of $\chi$ is defined as the integer \({\rm cod}(\chi)=|G:\ker\chi|/\chi(1)\). It was conjectured by G. Qian in \cite{Q} that, for every element $g$ of $G$, there exists an irreducible character \(\chi\) of $G$ such that \({\rm cod}(\chi)\) is a multiple of the order of $g$; the conjecture has been verified under the assumption that $G$ is solvable (\cite{Q}) or almost-simple (\cite{M}). In this paper, we prove that Qian's conjecture is true for every finite group whose Fitting subgroup is trivial, and we show that the analysis of the full conjecture can be reduced to groups having a solvable socle.
\end{abstract}

\maketitle

\section*{Introduction}
Let $G$ be a finite group, and let \(\irr G\) denote the set of the irreducible complex characters of $G$; given a character \(\chi\) in \(\irr G\), the \emph{codegree} of \(\chi\) is defined (following \cite{QWW}) as the integer \[{\rm {cod}}(\chi)=\dfrac{|G:\ker\chi|}{\chi(1)}.\] 

An interesting problem related to character codegrees was introduced by G. Qian in \cite{Q0}, and then formulated by the same author as a conjecture in \cite{Q}: it is asked whether, for every element $g$ of $G$, there exists $\chi\in\irr G$ such that \({\rm {cod}}(\chi)\) is divisible by the order of $g$. This conjecture is proved to be true in \cite{Q} under the assumption that $G$ is solvable.

A ``modular version" of the above conjecture is also considered, and proved true for finite solvable groups, by X. Chen and G. Navarro in \cite{CN}. As for non-solvable groups, Qian's conjecture has been verified by E. Giannelli for finite symmetric and alternating groups (\cite{G}); furthermore, S.Y. Madanha proves in \cite{M} that the conjecture is true for finite almost-simple groups as well.

The present note is a contribution in this framework. The following main result extends the validity of Qian's conjecture to finite groups whose Fitting subgroup is trivial.

\begin{ThmA} Let $G$ be a finite group whose Fitting subgroup is trivial, and let $g$ be an element of $G$. Then there exists \(\chi\in\irr G\) such that \({\rm cod}(\chi)\) is a multiple of the order of $g$.
\end{ThmA}

As a further step towards a possible proof of Qian's conjecture in full generality, we will also see that a minimal counterexample would be a group whose minimal normal subgroups are all abelian (Remark~\ref{remark}). 

In the following discussion we will freely use basic facts concerning character theory, for which we refer to \cite{Is}; also, every group will be tacitly assumed to be a finite group.

\section{Preliminaries}

We start by recalling some standard facts and establishing some notation in Remark~\ref{wreath}.

\begin{rem}\label{wreath}

Let $G$ be a group, and assume that $G$ has a unique minimal normal subgroup $M$; assume also that $M$ is non-solvable, thus $M=S_1\times\cdots\times S_n$ where the $S_i$ are pairwise isomorphic non-abelian simple groups.

Let \(\Omega=\{S_1,\ldots,S_n\}\), \(N=\norm G{S_1}\), and let \(T=\{t_1=1,\ldots, t_n\}\) be a right transversal for \(N\) in \(G\). The (transitive) action of $G$ by conjugation on \(\Omega\), i.e. the action of $G$ by right multiplication on the set \(\{Nt_i\mid i\in\{1,\ldots,n\}\}\), defines a homomorphism \(g\mapsto \sigma_g\) from \(G\) to \(\sym{\Omega}\cong\sym n\); moreover, for \(g\in G\) and \(i\in\{1,\ldots, n\}\), the element \(g_i=t_igt_{\sigma_g(i)}^{-1}\) lies in \(N\). 

Considering the factor group \(\overline{N}=N/\cent G{S_1}\) and adopting the bar convention, we see that \[g\mapsto(\overline{g_1},\ldots,\overline{g_n})\sigma_g\] defines an injective homomorphism from $G$ to the wreath product \(\Gamma=\aut{S_1}\wr\sym{n}\): this homomorphism is in fact the composition map of the injective homomorphism
$g \mapsto (g_1, \ldots, g_n)\sigma_g$ (see \cite[13.3]{CR})
with the natural homomorphism from $N\wr\sym{n}$ onto
$\overline{N}\wr\sym{n}$ (of course, here we are regarding $\overline N$ as a subgroup of \(\aut{S_1}\)), and the injectivity of this composition map is guaranteed by the fact that the normal core of \(\cent G{S_1}\) in $G$ is $\cent G M=1$. 

Identifying $G$ with a subgroup of \(\Gamma\), if \(\alpha_1\) is an irreducible character of \(S_1\) (or, more generally, of a subgroup \(X_1\) of \(\aut{S_1}\)) we will say that \(\alpha_1^{t_i}\) is the character of \(S_i\) (of \(X_1^{t_i}\)) corresponding to \(\alpha_1\).
\end{rem}

\smallskip
Our proof of Theorem~A relies on Lemma~\ref{monolithic}, concerning monolithic groups with a non-solvable socle; as a relevant preliminary ingredient, we gather some properties of characters of non-abelian simple groups in Lemma~\ref{simple} and Lemma~\ref{simple2}.

\begin{table}[h]
  \caption{Classical groups of Lie type in characteristic \(p\in\{2,3\}\).}
  \centering
  \resizebox{\textwidth}{!}{
    \begin{tabular}{cccc}
      \hline
      {Isomorphism type} & order & $\alpha(1)$ & Exponent of a Sylow $p$-subgroup\\
      \hline
      \\ \\
      \begin{tabular}{c}$A_1(q) \simeq \text{SL}_2(q),\; q = 2^{2f+1}$ \end{tabular} & $q (q^2-1)$ & $q-1$ & $2$ \\
      \\ \\
      \begin{tabular}{c}$A_1(q) \simeq \text{PSL}_2(q),\;q = p^{2f}$ \end{tabular} & $\dfrac{q (q^2-1)}{(2,q-1)}$ & {$q+1$} & $2,3$ \\
      \\ \\
      $A_2(q) \simeq \text{PSL}_3(q)$ & $\dfrac{q^3 (q^2-1)(q^3-1)}{(3,q-1)}$ & $q (q^2+q+1)$ & $2^2, 3$\\ 
      \\ \\
 $A_n(q) \simeq \text{PSL}_{n+1}(q),\;n \geq 3$ & $\dfrac{q^{\frac{n(n+1)}{2}} \prod_{i=1}^n (q^{i+1}-1)}{(n+1,q-1)}$ & $\dfrac{q(q^n-1)}{q-1}$ & \begin{tabular}{c}$\leq np$ \end{tabular} \\
      \\ \\    
  $^2A_2(q^2) \simeq \text{PSU}_{3}(q)$ & $\dfrac{q^3 (q^2-1)(q^3+1)}{(3,q+1)}$ & $q (q^2-q+1)$ & $2^2, 3$ \\
      \\ \\      
      \begin{tabular}{c}$^2A_n(q^2) \simeq \text{PSU}_{n+1}(q),\;n \geq 3$ \end{tabular} & $\dfrac{q^{\frac{n(n+1)}{2}} \prod_{i=1}^n (q^{i+1}-(-1)^{i+1})}{(n+1,q+1)}$ & $\dfrac{q(q^n-(-1)^n)}{q+1}$ & \begin{tabular}{c}$\leq np$ \end{tabular} \\
      \\ \\
      $\begin{matrix}  B_2(q)\simeq C_2(q) \simeq \text{PSp}_{4}(q)  \end{matrix}$ & $\dfrac{q^4 (q^2-1)(q^4-1)}{(2,q-1)}$  & $\dfrac{q(q^2+1)}{2}$ & $2^2, 3^2$ \\
      \\ \\
      $\begin{matrix} B_n(q) \simeq \Omega_{2n+1}(q),\;\;\\ C_n(q) \simeq \text{PSp}_{2n}(q), \;\;\end{matrix}$ $n \geq 3$ & $\dfrac{q^{n^2} \prod_{i=1}^n (q^{2i}-1)}{(2,q-1)}$ & $\dfrac{q(q^n+1)(q^{n-1}-1)}{2(q-1)}$ & $\begin{matrix} \leq (2n-1)p \end{matrix}$ \\
      \\ \\
      $D_n(q) \simeq \text{P}\Omega_{2n}^+(q),\;n \geq 4$ & $\dfrac{q^{n(n-1)}(q^n-1) \prod_{i=1}^{n-1} (q^{2i}-1)}{(4,q^n-1)}$ & $\dfrac{q(q^{n-2}+1)(q^n-1)}{q^2-1}$ & $\begin{matrix} \leq (2n-3)p \end{matrix}$ \\
      \\ \\
      $^2D_n(q^2) \simeq \text{P}\Omega_{2n}^-(q),\; n \geq 4$ & $\dfrac{q^{n(n-1)}(q^n+1) \prod_{i=1}^{n-1} (q^{2i}-1)}{(4,q^n+1)}$ & $\dfrac{q(q^{n-2}-1)(q^n+1)}{q^2-1}$ & $\begin{matrix} \leq (2n-3)p \end{matrix}$ \\
      \\ \\
      \hline
    \end{tabular}
  }
\end{table}

\begin{table}[h]

  \caption{Exceptional groups of Lie type in characteristic $p\in\{2,3\}$  (part I).}
  \centering
\resizebox{16cm}{!}{
   \begin{tabular}{ccccc}

  \hline
  Isomorphism type & Order &  Label of \(\alpha\) and \(\beta\) & $\alpha(1)$ and $\beta(1)$\\
  \hline
  \\
 $G_2(q)$ &$q^6\Phi_1^2\Phi_2^2\Phi_3\Phi_6$  &$\begin{matrix} \phi_{2,1} \\G_2[1]\end{matrix}$ & $\begin{matrix}\frac{1}{6} q\Phi_2^2\Phi_3\\  \frac{1}{6}q \Phi_1^2\Phi_6\end{matrix}$\\
 \medskip \\
  $^3D_4(q^3)$&$ q^{12}\Phi_1^2 \Phi_2^2\Phi_3^2\Phi_6^2\Phi_{12}$& $\begin{matrix} \phi_{1,3'} \\ \phi_{2,1}\end{matrix}$ & $\begin{matrix}q\Phi_{12}\\ \frac{1}{2} q^3\Phi_2^2\Phi_6^2\end{matrix}$\\
\medskip \\

$E_6(q)$&$q^{36}\Phi_1^6\Phi_2 ^4\Phi_3 ^3 \Phi_4^2\Phi_5\Phi_6 ^2\Phi_8\Phi_9\Phi_{12}$ &$\begin{matrix}  \phi_{64, 4}\\ D_{4},1\end{matrix}$ &$\begin{matrix} q^4\Phi^3_2\Phi^2_4\Phi^2_6\Phi_8\Phi_{12}\\ \frac 1{2} q^3\Phi_1^4 \Phi_3^2\Phi_5\Phi_9 \end{matrix}$    \\

\medskip \\
$^2E_6(q^2)$&$q^{36}\Phi_1^4\Phi_2^6\Phi_3^2\Phi_4^2\Phi_6^3\Phi_8\Phi_{10}\Phi_{12}\Phi_{18}$&$\begin{matrix}{\phi_{9,6}}'\\ ^2E_6[\theta]\end{matrix}$&$\begin{matrix}  q^6\Phi_3^2\Phi_6^3\Phi_{12}\Phi_{18}\\ \frac{1}{3} q^2\Phi_1^4\Phi_2^6\Phi_4^2\Phi_8\Phi_{10}\end{matrix}$ \\

\medskip \\
$E_7(q)$&$\begin{matrix}q^{63}\Phi_1^7\Phi_2^7\Phi_3^3\Phi_4^2\Phi_5\Phi_6^3\Phi_7\cdot\\\cdot\Phi_8\Phi_9\Phi_{10}\Phi_{12}\Phi_{14}\Phi_{18}\end{matrix}$& $\begin{matrix} E_6[\theta],1\\ \phi_{27,2}\end{matrix}$& $\begin{matrix} \frac{1}{3} q^7\Phi_1^6\Phi_2^6\Phi_4^2\Phi_5\Phi_7\Phi_8\Phi_{10}\Phi_{14}\\ q^2\Phi_3^2\Phi_6^2\Phi_9\Phi_{12}\Phi_{18} \end{matrix}$ \\

\\

\medskip 
$E_8(q)$&$\begin{matrix}q^{120}\Phi_1^8\Phi_2^8\Phi_3^4\Phi_4^4\Phi_5^2\Phi_6^4\Phi_7\Phi_8^2\Phi_9\cdot \\\cdot\Phi_{10}^2 \Phi_{12}^2\Phi_{14}\Phi_{15}\Phi_{18}\Phi_{20}\Phi_{24}\Phi_{30}\end{matrix}$& $\begin{matrix} \phi_{8,1}\\E_8[i] \end{matrix}$& $\begin{matrix}q \Phi_4^2\Phi_8\Phi_{12}\Phi_{20}\Phi_{24}\\\frac1{4} q^{16}\Phi_1^8\Phi_2^8\Phi_3^4\Phi_5^2\Phi_6^4\Phi_7\Phi_9\Phi_{10}^2\Phi_{14}\Phi_{15}\Phi_{18}\Phi_{30}  \end{matrix}$ \\

\medskip \\
$^2 F_4(q^2),\;q^2=2^{2f+1} >2$&$q^{24}\Phi_1^2\Phi_2^2 \Phi_4^2\Phi_8^2\Phi_{12}\Phi_{24}$& $\begin{matrix} \epsilon'\\\text{ cuspidal} \end{matrix}$& $\begin{matrix} q^2\Phi_{12}\Phi_{24}\\\frac1{3}q^4\Phi_1^2\Phi_2^2\Phi_4^2\Phi_8^2  \end{matrix}$ \\

\\
  \hline
\end{tabular}
}
\end{table}


\begin{table}[h]
  \caption{Exceptional groups of Lie type in characteristic $p\in\{2,3\}$ (part II). }
  \centering
   \resizebox{\textwidth}{!}{
\begin{tabular}{ccccc}
  \hline
  Isomorphism type & Order& Label of $\alpha$& $\alpha(1)$ & Exponent of a Sylow $p$-subgroup\\
  
  \hline
  \\
 $F_4(q)$&$q^{24}\Phi_1^4\Phi_2^4 \Phi_3^2\Phi_4^2\Phi_6^2\Phi_8\Phi_{12}$&$\phi_{4,1}$ &$\frac{1}{2} q \Phi_2^2\Phi_6^2\Phi_8$ & $2^4, 3^3$\\\
\medskip
\\
$^2G_2(q^2),\;q^2=3^{2f+1}>3$&$q^{6}\Phi_1\Phi_2 \Phi_4\Phi_{12}$& cuspidal & $\frac{1}{\sqrt 3}q\Phi_1\Phi_2\Phi_4$&$3^2$ \\
\medskip
\\
$^2B_2(q^2),\;q^2 = 2^{2f+1}>2$ &$q^{4}\Phi_1\Phi_2 \Phi_8$&$^2\text{B} _2[a]$& $\frac{1}{\sqrt 2}q\Phi_1\Phi_2$&$2^2$ \\
\\

  \hline
\end{tabular}}
\end{table}

\begin{lemma}\label{simple}

Let $S$ be a non-abelian simple group, and assume $S\not\cong\PSL{3^f}$  for any odd positive  integer $f$. Then there exist two distinct non-principal irreducible characters \(\alpha\) and \(\beta\) of \(S\), both having an extension to \(\aut S\), such that \(\dfrac{|S|}{\alpha(1)}\cdot\dfrac{|S|}{\beta(1)}\) is a multiple of the exponent of \(S\).
\end{lemma}

\begin{proof}
Suppose that $S$ is an alternating group ${\rm{Alt}}(n)$ where $n \geq 7$, or a sporadic simple group or the Tits group. According to Theorem~3 and Theorem~4 in \cite{BCLP}, there exist two non-principal irreducible characters $\alpha$ and $\beta$ of $S$ whose degrees are coprime, and both characters extend to $\text{Aut}(S)$. Consequently, \(\dfrac{|S|}{\alpha(1)}\cdot\dfrac{|S|}{\beta(1)}\) is a multiple of the order of $S$, thus a multiple of the exponent of $S$.

Consider next the case when $S$ is a simple group of Lie type (thus including ${\rm{Alt}}(5)$ and ${\rm{Alt}}(6)$), and denote by $p$ the characteristic of $S$. If $p > 3$, Theorem B in \cite{GRS} guarantees the existence of a non-principal irreducible character \(\alpha\) of \(S\) whose degree is not divisible by $p$ and which extends to $\text{Aut}(S)$; as for $\beta$, we can choose the Steinberg character of $S$ (whose degree is the full $p$-part of the order of $S$ and which, by \cite{S}, has as extension to $\text{Aut}(S)$ as well).  Also in this situation \(\dfrac{|S|}{\alpha(1)}\cdot\dfrac{|S|}{\beta(1)}\) is clearly a multiple of $|S|$.

To complete the proof, it remains to establish our claim for simple groups of Lie type in characteristic \(p\in\{2,3\}\). Assume first that $S$ is a classical group of Lie type, and take $\beta$ to be the Steinberg character of $S$: clearly $|S|/\beta(1)$ is a multiple of the $p'$-part of the order of any element $g$ in $S$. Therefore, our aim is to determine a non-principal irreducible character $\alpha$ of $S$ such that $|S|/\alpha(1)$ is divisible by the $p$-part of the order of any element in $S$ (in other words, by the exponent of a Sylow $p$-subgroup of $S$). Such a character $\alpha$ is described for each isomorphism type of $S$ in Table 1, which provides the degree of $\alpha$ and a bound for the exponent of a Sylow \(p\)-subgroup of $S$; the data of Table~1 are taken from Section~3 of \cite{M}, with the further remark that all characters listed there have an extension to \(\aut{S}\) by Theorem~2.4 and Theorem~2.5 in \cite{Ma} (for the first two rows, see also Theorem~A and Lemma~5.3(ii) of \cite{W}).

In Table~2 some exceptional groups of Lie type are considered. Here, for each group $S$, we list two irreducible characters $\alpha$ and $\beta$ that satisfy the conclusions of the statement. The data of Table~2 can be found in \cite[Section~13.9]{Ca}, and the extendability of the relevant characters is again ensured by Theorem~2.4 and Theorem~2.5 in \cite{Ma}.\newpage

Finally, we focus on the exceptional groups of Lie type listed in Table 3. Again we consider $\beta$ as the Steinberg character of the relevant group $S$, and $\alpha$ as the character appearing in the table. It is clear that the $p'$-part of $|S|$ divides  \(\dfrac{|S|}{\alpha(1)}\cdot\dfrac{|S|}{\beta(1)}\), and we also see that the exponent of a Sylow $p$-subgroup of $S$ divides the $p$-part of \(\dfrac{|S|}{\alpha(1)}\cdot\dfrac{|S|}{\beta(1)}\). The data of Table~3 are taken from \cite[Section~13.9]{Ca},  \cite[Section~4]{M} and, for what concerns the exponent of a Sylow \(p\) subgroup of \(F_4(q)\), from \cite[Theorem~3.1]{GZ}; the extendability of \(\alpha\) is ensured, as usual, by Theorem~2.4 and Theorem~2.5 in \cite{Ma}.
\end{proof}

\begin{lemma}\label{simple2}
Let $S$ be a non-abelian simple group such that \(S\not\cong\PSL{3^f}\) for any odd positive integer $f$, and let \(x\) be an element of \(S\). Then there exists a non-principal irreducible character \(\alpha\) of \(S\) which has an extension to \(\aut S\) and such that \(|S|/\alpha(1)\) is a multiple of the order of \(x\).
\end{lemma}
\begin{proof} Note first that, by the main theorem of \cite{M}, the statement is true whenever ${\rm Out}(S)$ is trivial, thus we may assume ${\rm Out}(S)\neq 1$.

Let us consider the case when $S$ is a sporadic simple group or the Tits group. According to the isomorphism type of $S$, in Table~1 of \cite{BCLP} it is possible to find two non-principal irreducible characters of $S$ that both extend to $\aut S$; it can be checked that, unless \(S\) is isomorphic to ${\rm Fi}_{22}$, one of those is suitable to be taken as a character $\alpha$ such that $|S|/\alpha(1)$ is a multiple of ${\rm o}(x)$. As for \(S\cong {\rm Fi}_{22}\), referring to the notation of \cite{ATLAS}, an appropriate character \(\alpha\) can be found in the set $\{\chi_2,\chi_{56}\}$.

Now, assume that $S$ is isomorphic to an alternating group ${\rm{Alt}}(n)$ for \(n\geq 7\). In this case the desired conclusion can be easily deduced from the proof of \cite[Theorem~A]{G}, where Qian's conjecture is established for symmetric and alternating groups; for the convenience of the reader, we sketch next the relevant argument. 

Consider the prime factorization \[{\rm o}(x)=2^k\cdot p_1^{k_1}\cdots p_t^{k_t}\] of the order of $x$, where $k\geq 0$ and \(k_i>0\) for every \(i\in\{1,\ldots, t\}\) (taking into account that the set of odd primes \(\{p_1,\ldots p_t\}\) can be empty). The proof of \cite[Theorem~A]{G} yields a non-principal irreducible character \(\alpha\) of $S$  such that \(|S|/\alpha(1)\) is a multiple of $2^{2k-1}\cdot p_1^{2k_1-1}\cdots p_t^{2k_t-1}$ if \(k\neq 0\), and of $p_1^{2k_1-1}\cdots p_t^{2k_t-1}$ if $k=0$ (hence, in any case, a multiple of ${\rm o}(x)$): for our purposes, it is then enough to check whether $\alpha$  has an extension to $\aut S\cong\sym n$ and, as we will see, this does happen in most cases. 

In fact, depending on the prime decomposition of ${\rm o}(x)$, the character $\alpha$ is chosen as an irreducible constituent of \(\chi_{S}\), where \(\chi\in\irr{\sym n}\) is the character associated to one of the following partitions: $\lambda=(n-1,1)$ or $\mu=(n-2,2)$ if $k=0$; $\nu=(2^k+1, 1^{n-2^k-1})$ if $k\neq 0$. Observe that $\lambda$, $\mu$ and $\nu$ are not self-associated, hence \(\chi_S\) lies in \(\irr S\) as we want, except for $\nu$ in the case when (\(k\neq 0\) and)   $n=2^{k+1}+1$. 
But in the latter case, still following the argument in the proof of \cite[Theorem~A]{G}, we get $p_1^{k_1}+\cdots + p_t^{k_t}\leq 2^k-1$, thus the largest prime power that divides ${\rm o}(x)$ is $2^k$; since $2^k$ is smaller than $n-1=2^{k+1}$, denoting by \(\chi^{\lambda}\in\irr{\sym n}\) the character associated to the partition $\lambda$, it turns out that $|\sym n|/\chi^{\lambda}(1)=2|S|/\chi^{\lambda}(1)$ is a multiple of $2^{2k-1}\cdot p_1^{2k_1-1}\cdots p_t^{2k_t-1}$. We deduce that $|S|/\chi^{\lambda}(1)$ is a multiple of $2^{2k-2}\cdot p_1^{2k_1-1}\cdots p_t^{2k_t-1}$, which is in turn a multiple of ${\rm o}(x)$ unless \(k=1\): but $k=1$ yields $n=5$, not our case, and the desired conclusion follows taking into account that \(\chi^{\lambda}_S\) lies in \(\irr S\).

Finally, let \(S\) be a simple group of Lie type (thus including ${\rm{Alt}}(5)$ and ${\rm{Alt}}(6)$). In this case, our claim is ensured by \cite[Theorem~5.1]{M} when $S\not\cong\PSL q$ for any prime power $q$. If $S\cong\PSL {p^f}$ for $p>3$, taking into account that the order of $x$ is either $p$ or a $p'$-number, we can define $\alpha$ as the character provided by \cite[Theorem~B]{GRS} or the Steinberg character of $S$, respectively. As for \(S\cong\SL{2^f}\), or \(S\cong\PSL{3^f}\) with an even \(f\), the character \(\alpha\) (of degree $2^f+(-1)^f$ or $3^f+1$, respectively) is provided by Theorem~A and Lemma~5.3(ii) of \cite{W} if ${\rm{o}}(x)=p$, or as the Steinberg character of $S$ otherwise.
\end{proof}

\begin{rem}\label{exception}
Note that any group $S\cong\PSL{3^f}$, where $f\geq 3$ is an odd positive integer, is a genuine exception to Lemma~\ref{simple} and Lemma~\ref{simple2}. In fact, it is well known that the degrees of the irreducible characters of $S$ are the integers in the set \(\{1,\;(3^f-1)/2,\; 3^f-1,\; 3^f,\; 3^f+1\}\) (see \cite{W}, for instance); recalling that the outer automorphism group of $S$ has order $2f$, and it is generated by a field automorphism $\phi$ of order $f$ and a diagonal automorphism $\overline\delta$ of order $2$, by  Lemma~4.1, Lemma~4.5 and  Lemma~4.6 of \cite{W} the two irreducible characters of degree $(3^f-1)/2$ are both invariant under \(\phi\) (hence they extend to \(S\langle\phi\rangle\)), but they are interchanged by \(\overline{\delta}\). Also, Lemma~5.2(i) and Lemma~5.3(iii) in \cite{W} show that $\langle\phi\rangle$ does not stabilize any irreducible character of $S$ whose degree is either $3^f-1$ or $3^f+1$; as a consequence, the only non-principal irreducible character of $S$ that has an extension to \(\aut S\) is the Steinberg character (of degree $3^f$). 
\end{rem}

Another key ingredient for the proof of Lemma~\ref{monolithic} will be the information, provided by Lemma~\ref{NotPSL} and Lemma~\ref{PSL}, on the extendability of certain irreducible characters in a monolithic group $G$ with non-solvable socle $M\cong S_1\times\cdots\times S_n\). For these lemmas and for Lemma~\ref{monolithic}, we will assume that an injective homomorphism from $G$ to \(\Gamma=\aut{S_1}\wr\sym n\) as described in Remark~\ref{wreath} has been preliminary fixed. 

\begin{lemma}\label{NotPSL}
Let $G$ be a group having a unique minimal normal subgroup $M$, and assume \linebreak $M=S_1\times\cdots\times S_n$, where the $S_i$ are pairwise isomorphic non-abelian simple groups. Let \(\alpha_1\) be an irreducible character of \(S_1\) which has an extension to \(\aut{S_1}\) and, for every \(i\in\{1,\ldots, n\}\), let \(\alpha_i\) be the corresponding character in \(\irr{S_i}\). Also, for a given \(h\in\{1,\ldots,n\}\), set \(M_1=S_1\times\cdots\times S_h\) and $M_2=S_{h+1}\times\cdots\times S_n\). Then the irreducible character \(\lambda=(\alpha_1\times\cdots\times\alpha_h)\times 1_{M_2}\) of \(M\) has an extension to its inertia subgroup \(I_G(\lambda)=\norm G{M_1}=\norm G{M_2}\). 
\end{lemma}

\begin{proof}
For \(i\in\{1,\ldots,n\}\), define \(A_i=\aut{S_i}\) and set $B_1=A_1\times\cdots\times A_h$, \(B_2=A_{h+1}\times\cdots\times A_n\), $B=B_1\times B_2\). Given an extension \(\widehat{\alpha_1}\) of \(\alpha_1\) to \(A_1\), let  \(\widehat{\alpha_i}\) be the corresponding character in \(\irr{A_i}\) and note that $\widehat{\lambda}=(\widehat{\alpha_1}\times\cdots\times\widehat{\alpha_{h}})\times 1_{B_2}\in\irr{B}\) is an extension of \(\lambda\). Since \(B\) is the base group of the wreath product \(\Gamma=\aut{S_1}\wr\sym{n}\), by Lemma~25.5(b) in \cite{Hu} there exists an extension \(\theta\) of \(\widehat{\lambda}\) to its inertia subgroup \(I_{\Gamma}(\widehat{\lambda})\). Now, viewing \(G\) as a subgroup of \(\Gamma\), an element $g=(\overline{g_1},\ldots,\overline{g_n})\sigma_g\in G\) lies in \(I_G(\lambda)\) if and only $\sigma_g$ lies in ${\rm{Stab}}_{\sym{n}}(\{1,\ldots, h\})={\rm{Stab}}_{\sym{n}}(\{h+1,\ldots, n\})$, which means that $g$ lies in \(\norm G{M_1}=\norm G{M_2}\). Since \(I_{\Gamma}(\widehat{\lambda})=B\,{\rm{Stab}}_{\sym{n}}(\{1,\ldots, h\})\) contains \(I_G(\lambda)\), we get that \(\theta_{I_G(\lambda)}\) is an extension of \(\lambda\), as wanted. 
\end{proof}


The following variation will take care of the exceptions to Lemma~\ref{simple}. After that, we will be in a position to prove Lemma~\ref{monolithic}.

\begin{lemma}\label{PSL}
Let $G$ be a group having a unique minimal normal subgroup $M$, and assume \linebreak $M=S_1\times\cdots\times S_n$, where the $S_i$ are all isomorphic to \(\PSL{3^f}\) for a suitable odd integer $f\geq 3$. For every \(i\in\{1,\ldots, n\}\), let \(\gamma_i\) be an irreducible character of degree $(3^f-1)/2$ of \(S_i\); also, fixing \(h\in\{1,\ldots,n\}\), set \(M_1=S_1\times\cdots\times S_h\) and $M_2=S_{h+1}\times\cdots\times S_n\). Then the irreducible character \(\lambda=(\gamma_1\times\cdots\times\gamma_h)\times 1_{M_2}\) of \(M\) has an extension to its inertia subgroup \(I_G(\lambda)\subseteq\norm G{M_1}=\norm G{M_2} \).
\end{lemma}

\begin{proof}
As above, for \(i\in\{1,\ldots,n\}\), define \(A_i=\aut{S_i}\) and set $B_1=A_1\times\cdots\times A_h$, \(B_2=A_{h+1}\times\cdots\times A_n\), $B=B_1\times B_2\). 

Recalling that we have preliminary fixed a right transversal \(\{t_1=1,\ldots,t_n\}\) of \(\norm G{S_1}\) in $G$, for \(i\in\{1,\ldots,n\}\) we define $F_i=(S_1\langle \phi_1\rangle)^{t_i}$, where \(\phi_1\) is a field automorphism of \(S_1\) having order \(f\): by Remark~\ref{exception}, we know that each of the \(\gamma_i\) has an extension \(\widehat{\gamma_i}\) to \(F_i\). Also, define \(U=F_1\times\cdots\times F_h\times B_2\). 

Note that $\widehat{\lambda}=(\widehat{\gamma_1}\times\cdots\times\widehat{\gamma_{h}})\times 1_{B_2}\in\irr{U}\) is an extension of \(\lambda\) and, still taking into account Remark~\ref{exception}, we have \(I_B(\widehat{\lambda})=I_B(\lambda)=U\); therefore \(\widehat{\lambda}^B\) is an irreducible character of \(B\), and in fact we have \(\widehat{\lambda}^B=(\widehat{\gamma_1}^{A_1}\times\cdots\times\widehat{\gamma_{h}}^{A_{h}})\times 1_{B_2}\). As above, \(B\) being the base group of the wreath product \(\Gamma=\aut{S_1}\wr\sym{n}\), \cite[Lemma~25.5(b)]{Hu} ensures that there exists an extension \(\theta\) of \(\widehat{\lambda}^B\) to the inertia subgroup \(I_{\Gamma}(\widehat{\lambda}^B)\). Now, the restriction of \(\theta\) to \(M\) is the sum of all the conjugates \(\lambda^b\) where \(b\) runs over a transversal for \(U\) in \(B\); in particular, recalling that \(U\) coincides with \(I_B(\lambda)\), every irreducible constituent of \(\theta_M\) appears with multiplicity~\(1\). 

Observe that if an element \(g=(\overline{g_1},\ldots,\overline{g_n})\sigma_g\) of $G\leq\Gamma$ lies in \(I_G(\lambda)\), then necessarily \(\sigma_g\in{\rm{Stab}}_{\sym{n}}(\{1,\ldots, h\})\). Thus, in particular, we have $I_G(\lambda)\subseteq \norm G{M_1}=\norm G{M_2}$. Since \(I_{\Gamma}(\widehat{\lambda}^B)=B\,{\rm{Stab}}_{\sym{n}}(\{1,\ldots, h\})=\norm\Gamma{M_1}\), we see that \(I_G(\lambda)\) is contained in \(I_{\Gamma}(\widehat{\lambda}^B)\), hence we can consider an irreducible constituent \(\psi\) of \(\theta_{I_G(\lambda)}\) lying over \(\lambda\). Now, \(\psi_M\) is a multiple of \(\lambda\) and \(\lambda\) appears as an irreducible constituent of \(\psi_M\) with multiplicity \(1\): as a consequence, \(\psi\in\irr{I_G(\lambda)}\) is an extension of \(\lambda\), and the proof is complete. 
\end{proof}

\begin{lemma} \label{monolithic} Let $G$ be a group having a unique minimal normal subgroup $M$, and assume \linebreak $M=S_1\times\cdots\times S_n$, where the $S_i$ are pairwise isomorphic non-abelian simple groups. Also, let $g$ be an element of $G$, and let \(r\) denote the order of \(gM\in G/M\). Then the following conclusions hold.
\begin{enumeratei}
\item If $S_1\not\cong\PSL{3^f}$ for any odd positive integer $f$, then there exists a non-principal character \(\lambda\in\irr M\) such that $\lambda$ has an extension to $I=I_G(\lambda)$, $g$ lies in $I$, and $|M|/\lambda(1)$ is a multiple of ${\rm{o}}(g^r)$.
\item If $S_1\cong\PSL{3^f}$ for some odd positive integer $f$, then there exist a non-principal character \(\lambda\in\irr M\) and a suitable $h\leq n$ such that $\lambda$ has an extension to $I=I_G(\lambda)$ and $g^{2^h}\in I$. Furthermore, $|M|/\lambda(1)$ is a multiple of $2^h\, {\rm{o}}(g^r)$. 
\end{enumeratei}
\end{lemma}

\begin{proof} 

Set $\Omega=\{S_1,\ldots,S_n\}$ and $K=\bigcap_{i=1}^n\norm G{S_i}$, so that $G/K$ is isomorphic to a transitive subgroup of $\sym\Omega\cong\sym n$: up to renumbering the elements of $\Omega$,  there exists a suitable positive integer \(h\leq n\) such that the set \(\{S_1,S_2,\ldots,S_h\}\) is an orbit for the action of \(\langle gK\rangle\) on \(\Omega\). As usual, define \(M_1=S_1\times\cdots\times S_h\) and \(M_2=S_{h+1}\times\cdots\times S_n\) (where $M_2$ is meant to be trivial if $h=n$). 

We start with an observation that will be useful for proving claim~(b), so, let us assume \(S_1\cong\PSL{3^f}\) for a suitable odd integer \(f\geq 3\); in what follows, we will consider the wreath product \(\Gamma=\aut{S_1}\wr\sym{n}\) and its subgroups \(U\), \(B\) as defined in Lemma~\ref{PSL}, and we recall that an injective homomorphism from $G$ to \(\Gamma\) as in Remark~\ref{wreath} is preliminary fixed. Also, we write \(\langle g\rangle=X\times Y\), where \(|X|\) is an odd number and \(|Y|\) is a power of \(2\). 
Consider the set $$\Delta=\{(\gamma_1\times\cdots\times\gamma_h)\times 1_{M_2}\in\irr M\;\;\mid\;\; \gamma_i\in\irr{S_i}{\text{ and }}\gamma_i(1)=(3^f-1)/2\}.$$ We see that both $X$ (which normalizes $M_1$) and the $2$-group $B/U$ act on \(\Delta\); moreover, $X$ acts on $B/U$, the orders of $X$ and $B/U$ are coprime, the action of $B/U$ on \(\Delta\) is transitive (in fact, regular) and we have $$(\eta^b)^x=(\eta^x)^{b^x}$$ for every \(\eta\in\Delta\), $b\in B$ and $x\in X$. Therefore, Glauberman's Lemma~13.8 in \cite{Is} yields that there exists an element \(\lambda_1\) of \(\Delta\) such that \(X\) lies in \(I_G(\lambda_1)\); this \(\lambda_1\) also has an extension to \(I_G(\lambda_1)\) by Lemma~\ref{PSL}. If we choose \(\eta=(\gamma_1\times\cdots\times\gamma_h)\times 1_{M_2}\) in \(\Delta\) such that the \(\gamma_i\) are all characters corresponding to \(\gamma_1\), then it is easy to see that \(I_{\Gamma}(\eta)\) lies in \(\norm{\Gamma}{M_1}\) with \(|\norm{\Gamma}{M_1}:I_{\Gamma}(\eta)|=2^h\); since there exists \(b\in B\subseteq\norm{\Gamma}{M_1}\) such that \(\lambda_1=\eta^b\), we clearly get \(|\norm{\Gamma}{M_1}:I_{\Gamma}(\lambda_1)|=2^h\) as well. But then, as $g$ lies in \(\norm{\Gamma}{M_1}\), we have \(|\langle g\rangle:\langle g\rangle\cap I_{G}(\lambda_1)|\leq |\norm{\Gamma}{M_1}:I_{\Gamma}(\lambda_1)|=2^h\). Taking into account that, as we just proved, the Hall $2'$-subgroup of \(\langle g\rangle\) is contained in \(I_G(\lambda_1)\), it follows that  \(|\langle g\rangle:\langle g\rangle\cap I_{G}(\lambda_1)|\) is in fact a divisor of~\(2^h\) and therefore \(g^{2^h}\in I_G(\lambda_1)\).

Next, it will also be useful to take into account the following remark, which holds for both (a) and (b) under the assumption that the action of \(\langle gK\rangle\) on \(\Omega\) is transitive (in other words, when \(h=n\) and \(gK\) is identified with an \(n\)-cycle in $\sym n$). Recalling that \(r\) denotes the order of \(gM\in G/M\), let us write \(g^r=(s_1, ... ,s_n)\in M\): we note that the orders of the \(s_i\in S_i\) are all the same, for \(i\in\{1,\ldots, n\}\). In fact, write \(g=(\overline{g_1}, \ldots ,\overline{g_n})\sigma_g\) as an element of the wreath product \(\Gamma=\aut{S_1}\wr\sym{n}\). Conjugating $g^r$ with $g$, we get $(s_n^{g_n},s_1^{g_1}, \ldots,s_{n-1}^{g_{n-1}})\). This is clearly the same as \(g^r\), so in particular \(s_j=s_{j-1}^{g_{j-1}}\) for every $j\in\{2,\ldots,n\}$ and we get the desired property. As a consequence, the order of \(g^r\) is in fact the order of an element of \(S_1\). 

\smallskip
We can now work toward a proof of (a) and (b), and we first treat the case when 
the action of $\langle gK\rangle$ on $\Omega$ is \emph{not} transitive, so that we have $1\leq h< n$. 

If \(S_1\) is not isomorphic to \(\PSL{3^f}\) for any odd positive integer $f$, then Lemma~\ref{simple} yields the existence of two distinct non-principal characters \(\alpha_1,\,\beta_1\in\irr {S_1}\), both having an extension to \(\aut{S_1}\), such that \(\dfrac{|S_1|}{\alpha_1(1)}\cdot\dfrac{|S_1|}{\beta_1(1)}\) is a multiple of \({\rm{exp}}(S_1)\). Denoting by \(\alpha_i\) and \(\beta_i\) the characters of \(S_i\) corresponding to \(\alpha_1\) and \(\beta_1\) for \(i\in\{1,\ldots,n\}\), Lemma~\ref{NotPSL} yields that \(\lambda_1=\alpha_1\times\cdots\times\alpha_h\times 1_{M_2}\) and \(\lambda_2=1_{M_1}\times\beta_{h+1}\times\cdots\times\beta_n\) both extend to their inertia subgroup $I=\norm G{M_1}=\norm G{M_2}\). Define now \(\lambda=\lambda_1\lambda_2\in\irr M\): the inertia subgroup of \(\lambda\) in \(G\), which is easily seen to be $I$, contains the element $g$ and, by \cite[Theorem~6.16]{Is}, \(\lambda\) has an extension to \(I\). Moreover, we get \(\lambda(1)=\alpha_1(1)^h\beta_1(1)^{n-h}\), therefore \(|M|/\lambda(1)\) is certainly a multiple of \({\rm{exp}}(S_1)\) and claim (a) in the non-transitive case immediately follows.

On the other hand, if \(S_1\cong\PSL{3^f}\) for a suitable odd integer \(f\geq 3\), then we consider a character \(\lambda_1\in\irr M\) as in the second paragraph of this proof: so, \(\lambda_1\) has an extension to \(I_G(\lambda_1)\) and \(g^{2^h}\in I_G(\lambda_1)\). Also, define \(\beta_i\) as the Steinberg character of \(S_i\), set \(\lambda_2=1_{M_1}\times\beta_{h+1}\times\cdots\times\beta_n\) and observe that \(\lambda_2\), whose degree is \(3^{f(n-h)}\), extends to \(I_G(\lambda_2)=\norm G {M_1}=\norm G{M_2}\) by Lemma~\ref{NotPSL}. Set now \(\lambda=\lambda_1\lambda_2\); the inertia subgroup of \(\lambda\) turns out to be \(I=I_G({\lambda_1})\), and \(\lambda\) extends to $I$ again by Theorem~6.16 of \cite{Is}. Recalling that \(|S_i|=\dfrac{(3^f-1)\cdot 3^f\cdot(3^f+1)}{2}\), we have \[\dfrac{M}{\lambda(1)}=2^h\cdot\dfrac{|S_1|}{3^f-1}\cdots\dfrac{|S_h|}{3^f-1}\cdot\dfrac{|S_{h+1}|}{3^f}\cdots\dfrac{|S_{n}|}{3^f},\] which is certainly a multiple of \(2^h\cdot{\rm{exp}}(S_1)=2^h\cdot\dfrac{(3^f-1)\cdot 3\cdot (3^f+1)}{4}\) and, in particular, of \(2^h\,{\rm o}(g^r)\). Claim~(b) is thus proved in the non-transitive case.  

\smallskip
We move next to the case when the action of \(\langle gK\rangle\) on \(\Omega\) is transitive; as previously observed, in this case the order of \(g^r\) is in fact the order of an element of \(S_1\). 

If \(S_1\not\cong\PSL{3^f}\) for any odd integer \(f\geq 3\) then, by Lemma~\ref{simple2}, there exists an irreducible character \(\alpha_1\in\irr{S_1}\) such that \(|S_1|/\alpha_1(1)\) is a multiple of ${\rm{o}}(g^r)$ and \(\alpha_1\) extends to \(\aut{S_1}\); therefore, by Lemma~\ref{NotPSL}, the character \(\lambda=\alpha_1\times\cdots\times\alpha_n\) extends to \(I_G(\lambda)=G\) and clearly satisfies the conclusions of claim~(a). 

It remains to consider the case when \(S\cong\PSL{3^f}\) for an odd \(f\geq 3\) and the action of $\langle gK\rangle$ on $\Omega$ is transitive. Since \({\rm o}(g^r)\) is the order of an element of $S_1$, then it is either $3$ or a number coprime to $3$. For the former case we can consider a character \(\lambda_1\) as in the second paragraph of this proof (here $h=n$), whereas in the latter case we define \(\lambda_1\) as the direct product of the Steinberg characters of the \(S_i\), for \(i\in\{1,\ldots,n\}\), which extends to \(I=I_G(\lambda_1)\) by Lemma~\ref{NotPSL}. It can be easily checked that the conclusions of claim~(b) are satisfied by this \(\lambda_1\), so the proof is complete.
\end{proof}

\section{Proof of Theorem A}

Note that the conclusion of Claim~(a) in Lemma~\ref{monolithic} is stronger than that of Claim~(b); in fact, the former is just the latter with the additional property that $h=0$. In other words, Claim~(b) holds for any isomorphism type of $S_1$, and this is what will be relevant henceforth.

We are ready to prove Theorem~A, that we state again.

\begin{ThmA} Let $G$ be a group whose Fitting subgroup is trivial, and let $g$ be an element of $G$. Then there exists \(\chi\in\irr G\) such that \({\rm cod}(\chi)\) is a multiple of the order of $g$.
\end{ThmA}

\begin{proof}
Since the group $G$ has a trivial Fitting subgroup, the generalized Fitting subgroup $E$ of $G$ is the socle of $G$, thus $E=M_1\times\cdots\times M_k$ where the \(M_j\) are non-solvable minimal normal subgroups of $G$. For every \(j\) in \(\{1,\ldots,k\}\), \(M_j\) is in turn the direct product of pairwise isomorphic  non-abelian simple groups, and we denote by $n_j$ the composition length of $M_j$ (i.e. the number of simple direct factors appearing in this direct decomposition of $M_j$). 

Now, set \(C_j=\cent G{M_j}\) and denote by $V_j$ the product of all the $M_\ell$ for \(\ell\in\{1,\ldots, k\}-\{j\}\) (in particular, \(V_j\subseteq C_j\)); the factor group $\overline{G}_j=G/C_j$ has $\overline{M_j}$ as its unique minimal normal subgroup, thus we can apply Lemma~\ref{monolithic} to $\overline{G}_j$ with respect to the element \(gC_j\), and choose a character \(\overline{\lambda}_j\in\irr{\overline{M_j}}\) with a corresponding non-negative integer $h_j\leq n_j$ as in Lemma~\ref{monolithic}(b). 
Note that each \(\overline{\lambda}_j\) can be regarded by inflation as a character of $M_j\times C_j$ whose kernel contains $C_j$, hence there exists a unique \(\lambda_j\in\irr{M_j}\) such that \(\overline{\lambda}_j=\lambda_j\times 1_{C_j}\); given that, we define \(\lambda=\lambda_1\times\cdots\times\lambda_k\in\irr E\).

We know that the character \(\overline{\lambda}_j\) extends to $I_G(\overline{\lambda}_j)=I_G(\lambda_j)$, therefore \(\lambda_j\times 1_{V_j}\in\irr{E}\) extends to \(I_G(\lambda_j)\) as well. In particular, each \(\lambda_j\times 1_{V_j}\) has an extension \(\widehat{\lambda}_j\) to \(I=I_G(\lambda)=\bigcap_{s=1}^{k}I_G(\lambda_s)\), and it is easy to check that the product  \(\psi=\prod_{s=1}^k\widehat{\lambda}_s\) is an extension of \(\lambda\) to $I$. Furthermore, defining $h=h_1+\cdots + h_k$ and recalling that we have \(g^{2^{h_j}}\in I_G(\lambda_j)\) for every $j\in\{1,\ldots, k\}$, we get $g^{2^h}\in I$. 

Finally, set \(\chi=\psi^G\in\irr G\) and note that \(\chi\) is a faithful character of $G$, because \[\ker\chi\cap E\leq\ker\psi\cap E=\ker{\psi_E}=\ker{\lambda}=1,\] and a normal subgroup of $G$ which intersects $E$ trivially is necessarily trivial.

We are ready to conclude the proof. We get \[{\rm{cod}}(\chi)=\dfrac{|G|}{\chi(1)}=\dfrac{|G|}{|G:I|\psi(1)}=\dfrac{|I|}{|E|}\cdot\dfrac{|E|}{\psi(1)}\] and, denoting by \(r=|\langle g\rangle E/E|\) the order of \(gE\in G/E\), \[\dfrac{|I|}{|E|}=\dfrac{|I|}{|\langle g\rangle E|}\cdot r=\dfrac{|I:\langle g\rangle E\cap I|}{|\langle g\rangle E:\langle g\rangle E\cap I|}\cdot r=\dfrac{|I:\langle g\rangle E\cap I|}{|\langle g\rangle :\langle g\rangle\cap I|}\cdot r\,.\] In order to prove that ${\rm{cod}}(\chi)$ is a multiple of ${\rm o}(g)$, taking into account that \(|\langle g\rangle :\langle g\rangle\cap I|\) is a divisor of \(2^h\), it will then suffice to show that \(|E|/\psi(1)\) is a multiple of \(2^h\,{\rm o}(g^r)\).

In fact, for \(j\in\{1,\ldots,k\}\), consider $\overline{G}_j=G/C_j$, and denote by \(r_j\) the order of \(\overline{g}\overline{M_j}\) in \(\overline{G}_j/\overline{M_j}\). Clearly all the \(r_j\) are divisors of \(r\); since, for every \(j\in\{1,\ldots, t\}\), \(|M_j|/\lambda_j(1)\) is a multiple of \(2^{h_j}\,{\rm{o}}(\overline{g}^{r_j})\)  by Lemma~\ref{monolithic}, we see that \(|M_j|/\lambda_j(1)\) is a multiple of \(2^{h_j}\,{\rm{o}}(\overline{g}^{r})\) as well. Now, \[\dfrac{|E|}{\psi(1)}=\dfrac{|E|}{\lambda_1(1)\cdots\lambda_k(1)}=\dfrac{|M_1|}{\lambda_1(1)}\cdots\dfrac{|M_k|}{\lambda_k(1)}\] is a multiple of \(2^h\,{\rm{o}}(g^rC_1)\cdots{\rm{o}}(g^rC_k)\). Recalling that the map \(x\mapsto(xC_1,\ldots,xC_k)\) is an injective homomorphism from $G$ to \(G/C_1\times\cdots\times G/C_k\), it follows that the least common multiple of \({\rm{o}}(g^rC_1),\ldots, {\rm{o}}(g^rC_k)\) equals \({\rm{o}}(g^r)\), and the desired conclusion follows.
\end{proof}

\section{A reduction}

We conclude this note observing that Qian's conjecture can be reduced to groups with a solvable socle.

\begin{rem}\label{remark}
Assume that the group $G$ is a minimal counterexample to the conjecture stated in the Introduction; then we claim that $G$ does not have any non-solvable minimal normal subgroup. 

For a proof by contradiction, denote by $M$ a non-abelian minimal normal subgroup of $G$, set $C=\cent G M$, and observe that the factor group $\overline{G}=G/C$ is a monolithic group whose socle is $\overline{M}\cong M$. Therefore, for a fixed element $g$ of $G$, we can apply Lemma~\ref{monolithic} with respect to ${\overline g}=gC$ and obtain what follows: there exists a non-principal character \(\overline{\lambda}\in\irr{\overline{M}}\) and a non-negative integer \(h\) (not exceeding the composition length of $M$) such that $\overline{\lambda}$ has an extension to $\overline{I}=I_{\overline{G}}(\overline{\lambda})$, $\overline{g}^{2^h}=g^{2^h}C$ lies in $\overline{I}$, and $|\overline{M}|/\overline{\lambda}(1)$ is a multiple of $2^h\,{\rm{o}}(\overline{g}^{\overline{r}})$ where $\overline{r}$ is the order of $\overline{g}\overline{M}$ in $\overline{G}/\overline{M}$. By inflation, $\overline{\lambda}$ can be viewed as a character of $M\times C$ and, as such, it is of the form $\lambda\times 1_C$ for a suitable \(\lambda\in\irr M\); clearly, we have $I_G(\lambda)=I_G(\overline{\lambda})=I$ (hence \(g^{2^h}\in I\)) and $|\overline{M}|/\overline{\lambda}(1)=M/\lambda(1)$. Observe also that, if \(r\) denotes the order of \(gM\) in \(G/M\), then \(r\) is a multiple of \(\overline{r}\) and therefore \({\rm{o}}(\overline{g}^{\overline{r}})\) is a multiple of \({\rm o}(\overline{g}^r)\); as the map \(x\mapsto \overline{x}\) is an isomorphism of \(M\) to \(\overline{M}\), we get \({\rm o}(\overline{g}^r)={\rm o}(\overline{g^r})={\rm o}({g}^r)\).

Now, we know that \(\lambda\) has an extension \(\widehat{\lambda}\) to $I$ such that \(\ker{\widehat{\lambda}}\) contains $C$; moreover, the minimality of $G$ yields that there exists \(\xi\in\irr{I/M}\) such that \(|I/M:\ker\xi|/\xi(1)\) is a multiple of \({\rm o}(g^{2^h}M)=r/{\rm{gcd}}(2^h,r)\). Define \(\psi\) as \(\widehat{\lambda}\xi\), which is in \(\irr I\) by Gallagher's Theorem, and \(\chi=\psi^G\): by Clifford Correspondence we have \(\chi\in\irr G\), and we claim that \({\rm cod}(\chi)\) is a multiple of the order of \(g\). It will follow that $G$ is not a counterexample to Qian's conjecture, so we have a contradiction. 

In fact, \[{\rm cod}(\chi)=\dfrac{|G:\ker\chi|}{\chi(1)}=\dfrac{1}{\ker\chi}\cdot\dfrac{|I|}{\psi(1)}=\dfrac{1}{\ker\chi}\cdot\dfrac{|I/M|}{\xi(1)}\cdot\dfrac{|M|}{\lambda(1)}.\] Since \(|I/M:\ker\xi|/\xi(1)\) is a multiple of \(r/{\rm{gcd}}(2^h,r)\) and $|M|/\lambda(1)$ is a multiple of \(2^h\,{\rm o}(g^r)\), it will be enough to show that \(\ker\chi\) is contained in \(\ker\xi\): this can be deduced by the fact that \(\ker\chi\) is a normal subgroup of \(G\) intersecting \(M\) trivially, hence \(\ker\chi\subseteq C\cap\ker\psi=\ker{(\widehat{\lambda}\xi)_C}=\ker{\xi_C}\) (recall that $\ker{\widehat{\lambda}}$ contains $C$) and the argument is complete.



\end{rem}

\bigskip
{\bf Acknowledgment.} This work has been done during a visit of the first and third authors at Dipartimento di Matematica e Informatica (DIMAI) of Universit\` a degli Studi di Firenze. They wish to thank DIMAI for the hospitality. 

Also, the authors are grateful to S. Dolfi for helpful discussions on the subject of this paper.

\end{document}